\documentclass[journal]{IEEEtran}

\ifCLASSINFOpdf
\else
\fi
\usepackage{enumitem}
\usepackage{amsmath}
\usepackage{amsthm}
\usepackage{amssymb}
\usepackage{algorithm}
\usepackage{algpseudocode}
\usepackage{graphicx}

\newcommand{\V}{\mathcal{V}}
\newcommand{\T}{\mathcal{T}}
\newcommand{\F}{\mathcal{F}}
\newcommand{\J}{\mathcal{J}}
\newcommand{\bigO}{\mathcal{O}}

\newtheorem{theorem}{Theorem}[section]

\newcommand{\abs}[1]{|#1|}

\newcommand{\etal}{{\em et~al.~}}

\begin{document}
\title{Electric Vehicle Valet}
%
%
%

\author{Ali Khodabakhsh, Orestis Papadigenopoulos, Jannik Matuschke,\\ Jimmy Horn, Evdokia Nikolova, Emmanouil Pountourakis
}

%

\maketitle

\begin{abstract}
We propose a novel way to use Electric Vehicles (EVs) as dynamic mobile energy storage with the goal to support grid balancing during peak load times.  EVs seeking parking in a busy/expensive inner city area, can get free parking with a valet company in exchange for being utilized for grid support. The valet company would have an agreement with the local utility company to receive varying rewards for discharging EVs at designated times and locations of need (say, where power lines are congested).
Given vehicle availabilities, the valet company would compute an optimal schedule of which vehicle to utilize where and when so as to maximize rewards collected.  Our contributions are a detailed description of this new concept along with supporting theory to bring it to fruition.  On the theory side, we provide new hardness results, as well as efficient algorithms with provable performance guarantees that we also test empirically. 
\end{abstract}


%
\IEEEpeerreviewmaketitle

\section{Introduction}
%
%
%
%

\IEEEPARstart{W}{e} propose a new concept on how to use Electric Vehicles (EVs) as dynamic mobile energy storage with the goal to support grid balancing during peak load times.  During a given day, EV owners turn in their car to a central planner that manages charging and discharging.  Prior to that (say at the beginning of the day) the EV owners provide an availability schedule for their car (e.g., they will not need their car from 9 am to 5 pm while they are at work or from noon to 3 pm while they are at lunch and yoga class).  Each car is assumed to arrive fully charged so it can be discharged as soon as it is turned in (we will relax this assumption later).  The central planner utilizes valets to take the cars to different discharging locations around the city, and subsequently charge the cars.  In particular, the planner wants to compute a schedule for when to discharge available cars, and which valet to assign to which car in a given hour.

A few months back, one of the authors was looking for parking in downtown Boston to visit the Boston Aquarium and faced with the steep rate of \$40 for three hours of parking in the adjacent parking lot, started thinking how parking for electric vehicles could be made free, or even be rewarded, if it were coupled with utilizing the EV batteries as dynamic electricity storage that could be shifted between locations of need.  

In exchange for free parking (and possibly a reward) the EV owners agree to the valet service discharging their vehicle once or twice during their parking time.  
The valet drives the EV to nearby locations that are designated by the local electric utility company as congested or overloaded in the sense that the demand for power is straining the ability of the local power line to supply it.  The EV is discharged at that location, providing additional needed power, and then driven to a connection port at a non-congested nearby power line to be recharged.  At the end, the electric vehicle is returned to the owner fully charged. 

How can such a valet service be made operational?  The utility company would provide connection ports with varying rewards that support grid balancing during peak load times
through the discharging of EVs in designated locations at designated times. 
The valet service would then compute an optimal schedule of what EVs to discharge where and at what time.

In this paper, we present in detail the concept of EV Valet and then describe the theory towards making it operational, which includes novel NP-hardness results and several approximation algorithms.  

\vspace{0.5em}
\noindent {\bf EV Valet: Description of Concept.}
Our EV Valet concept (henceforth, {\em valet} for short) is to use valets to connect EV owners to locations of power supply need.  The first step consists in the utility company deciding where it needs help with extra power supply and creating connection ports at these locations.  It then announces rewards for these ports.  The rewards would capture the level of congestion and extend the life of supported power lines and equipment (in essence, reducing the cost and delaying the timing of upgrades). 
The assumption is that the valet company would operate in an area with regular congestion that can support sufficient rewards. 
Determining the optimal rewards is a topic we plan to pursue in future work, and is out of scope of the current paper. 

For charging, the utility company would offer a small number of free or low-cost charging locations with many ports in nearby uncongested (e.g., suburban or large commercial) areas.  If the charging is not free, we can assume it is free by redefining the reward to be the difference in price between the discharging location and the charging location. Additionally, the charging locations would serve as overflow parking.

There are two ways to set up the valet company: 1) As a third-party valet company whose goal is to maximize its own profit; 2) As part of the utility company, which has its own valets, and essentially develops the valet company in-house at cost.  We focus on the third-party approach because it is more complex than the in-house approach and 
we feel that profit maximization promotes stronger incentive for optimal scheduling.  

A key advantage of our idea 
is that it does not lead to a zero-sum outcome.  We specifically set it up so that everyone benefits, with a principal goal to make the {\em whole system more efficient} and {\em increase social welfare.}  

We now have three parties that participate in this process: the utility company, the valet company and the customers (EV owners).  Each party has to act reasonably (i.e., to take only a portion of the overall benefit) to make it worthwhile for the other parties to participate.  For this process to be sustainable, we take as a given a local electricity network with overloads frequent enough 
for the valet company to be profitable.  The utility company would offer rewards that are a percentage (say, 25\% to 50\%) of what its cost would be if the valet were not present.  In that case, the utility cost would be calculated as the cost to buy the extra needed  power and deal with overloads both in the short and long term.  

Once the utility announces the rewards (assume this is done a day or more in advance), the valet's job is to compute an optimal schedule for collecting the rewards, given the available cars with their variable schedules. Developing the necessary theory for this is our main technical contribution, which  proposes algorithms with provable worst-case guarantees.  

To entice EV owners to participate in its service, the valet has to provide appropriate incentives.  These incentives are in the simplest case free parking in a high cost parking area like downtown Boston, which could be as much as \$10/hour or more. 
The valet company might also give a percentage of the collected rewards as additional incentive if needed. 
The valet company also has to account for its overhead cost in valet employees and utilized parking lots.   
Note, the valet operation would have a smaller inner city parking lot for vehicle drop-offs and pick-ups, and a larger parking lot further out near an uncongested utility feeder, where the free or low cost charging would occur. 
 
The customers' benefit would be free parking, a portion of the collected rewards, if offered, and reduction of the overall cost of their utility company.  The latter is tied to the larger social benefit of lowering everyone's utility bills, including their own. 
The parking savings should far outweigh the wear and tear on the EV batteries, which we estimate at about \$5 per charging cycle.\footnote{This estimate is based on a 5 year replacement of a \$4,000 battery system, with 150 charges per year.}  In contrast, the valet service would provide upwards of \$20 in value per charging cycle through the offset of the customer's parking costs.  
This is similar to Uber/Lyft drivers who  trade-off the wear and tear on their car for monetary compensation.
Furthermore, the valet company could allow customers to set a limit of one or two discharges per day.  

The goal of the algorithms we develop in this paper is to optimize how well the valets collect the rewards. The better they do it, the bigger the benefit for each party.  In the event of a missed reward, (i) the utility would suffer an opportunity cost in the missed congestion reduction, 
(ii) the customer would have to pay full price for parking, and (iii) the valet would lose profit. 

We note that a positive externality of the valet concept is that it also benefits the overall social welfare by improving parking availability for others outside of the service. 

\vspace{0.5em}
\noindent
{\bf Our Contributions.}
In summary, our contributions are:


\begin{enumerate}
\item We propose the novel EV Valet concept as dynamic mobile energy storage with the goal to supply grid load balancing during peak load times. 
\item We set up a mathematical model for optimizing the discharging schedules for the EV valet service.
\item We prove that this EV valet problem is strongly NP-hard. In other words, unless $\textbf{P=NP}$, not only is there no polynomial-time algorithm for solving the problem optimally, but also there is no algorithm that can produce an assignment of reward greater than $1-\epsilon$ of the optimal, in time polynomial in the input size and in $1/{\epsilon}$.
\item Despite the computational hardness of the general setting, we obtain the optimal solution for 4 special cases of the problem, when i)~we have access to super-fast chargers, ii)~we have a single car, iii)~the number of EVs is constant, and iv)~all EVs have the same availability and constant charging time.
\item For the general setting, we give two approximation algorithms, namely \textsc{greedy} and \textsc{randomized rounding}, and we prove that they provide (worst-case) performance guarantees of $1/3$ and $1-1/e$, respectively.\footnote{An $\alpha$-approximation algorithm for a maximization problem is an algorithm that runs in polynomial time and returns a solution whose value is at least $\alpha$ times the optimal value.}
\end{enumerate}
The rest of this paper is devoted to the theoretical contributions 2-5 above. 

\section{Related work}
\label{sec:related}
EV charging and discharging has previously been studied in the context of power loss minimization \cite{singh2010influence}, frequency regulation \cite{liu2013decentralized}, voltage regulation \cite{singh2010influence,yong2015bi}, peak shaving \cite{hussain2018mobility,zhang2017real}, and supporting renewable energy sources \cite{caramanis2009management}. We refer the reader to a recent survey by Amjad \etal \cite{amjad2018review} on various optimization approaches and objectives employed for EV charging. Our goal is most closely related to peak shaving, however the above-mentioned work on that is very different from ours in that 
it either does not consider discharging \cite{zhang2017real}, or does not incentivize the EV owners \cite{hussain2018mobility}.

A convex optimization method for optimal scheduling of EV charging and discharging is proposed by He \etal \cite{he2012optimal}, where the authors aim to coordinate the charging and discharging of EVs to minimize the total cost of all vehicles. The authors consider a unique arrival and departure time for each EV; however, they ignore the spatial variation of the price and assume that the electricity price at a time instant is the same regardless of the charging location. 
In comparison, we consider a more general availability model where each EV can leave the parking as many times as desired. 
Also, our model does not require any assumption on the price function. Most importantly, our model is very different conceptually in its objective to balance the grid compared to the objective of minimizing the charging cost of EVs.

Hutson \etal \cite{hutson2008intelligent} propose an intelligent scheduling of EVs in a parking lot, so as to maximize profits by discharging EVs at the times when the market power price is high, and charging when the price is low.  The spatial variation of the price is not considered in their model. 
Finally, they utilize a particle swarm optimization approach that does not provide any guarantee and, indeed, it can suffer from premature convergence.  In contrast, we provide guarantees. 
But more importantly, our goal is to alleviate congestion for the utility and not to collect the highest market power price. Sometimes these may coincide but are not necessarily related in general.  Additionally, we charge and discharge EVs in multiple locations in a geographic area as opposed to one parking lot.  

An online auction framework for EV charging is proposeded by Xiang \etal \cite{xiang2015auc2charge}, where in a large parking lot, every spot is equipped with a charging point, and the EV users submit bids on their charging demands. The parking lot then decides on the allocation and pricing based on the collected bids. It is shown that the proposed mechanism is truthful and individually rational, while approximately maximizing social welfare. Vehicle to grid (V2G) discharging is not considered in the proposed market mechanism, while effective discharging is our main objective.

A combination of an autonomous parking system with EV charging is studied by Timpner and Wolf \cite{timpner2014design}, with the goal of scheduling the charging times of autonomous vehicles on a limited number of charging stations in a parking lot. The difference with our valet model is that they consider homogeneous charging stations, unidirectional flow of energy (no V2G), and charging station's utilization as the objective, while in our paper heterogeneous stations in different locations, EV discharging (V2G) and reward collection are critical. Furthermore, the authors propose 5 different scheduling algorithms, but no theoretical guarantee is given with respect to their stated objective, whereas we provide theoretical guarantees. 

From an algorithmic perspective, most existing work is either based on mixed-integer programs \cite{jin2013optimizing}, which cannot be solved efficiently for large instances, or heuristics without any optimality guarantees. These heuristics include genetic algorithms \cite{vaya2012centralized}, particle swarm optimization \cite{hutson2008intelligent,saber2009optimization,soares2011optimal}, and ant colony optimization \cite{yang2015load}. In contrast, we exploit techniques from theoretical computer science, in particular we provide novel hardness results and approximation algorithms, as well as adapt an algorithm for an interval scheduling problem~\cite{bhatia2007algorithmic}, to provide efficient algorithms with rigorous performance guarantees.
\section{Preliminaries}
\label{sec:formulation}
In this section we propose our mathematical model for optimal discharge scheduling of EVs by the valet service, and then study the hardness of the proposed optimization problem.
\subsection{Problem definition}
We consider a set $\V = \{1,2,...,m\}$ of {\em vehicles} and a discrete time horizon $\T=[T]=\{1,2,\dots,T\}$. Every vehicle $i \in \V$ is associated with a set $\T_i \subseteq \T$ of {\em availabilities}, i.e. times where it can be discharged, and a {\em charging time} $C_i \in \mathbb{N}$, that is, the time the vehicle needs to recharge before being ready to be discharged again. We say that a vehicle $i \in \V$ is {\em available} at time $t$, when $t \in \T_i$. We also consider a set $\J=\{1, 2, \dots, n\}$ of {\em stations} that can be used for discharging a vehicle. At any time $t \in \T$, any available vehicle $i \in \V$ can be discharged at some station $j \in \J$, and obtain {\em reward}  $p_{j,t} \in \mathbb{R}$. 
Notice that the reward depends on both the station and time of discharging, which captures the utility's localized congestion costs. Moreover, a reward can be negative, modeling in this way the fact that the benefit to the utility company at a specific station/time can be smaller than the cost of charging. Our objective is to find a {\em feasible} assignment of vehicles to stations/times that maximizes the total collected reward. In any feasible assignment, the following conditions must hold: (a) Every station can be assigned to at most one vehicle $i \in \V$ at every time $t \in \T$ such that $t \in \T_i$. (b) Every vehicle $i \in \V$ can be associated with at most one station $j \in \J$ at any time $t \in \T_i$. (c) If a vehicle $i \in \V$ has been discharged at some station at time $t$, the same vehicle cannot be discharged again during the time interval $[t+1 , t + C_i]$, as it needs $C_i$ time units to get charged. To make sure that the valet service returns the car fully charged, we can simply prune the availabilities in the pre-processing phase, preventing the car to be discharged in its last $C_i$ available slots. In addition, we can relax the assumption on fully-charged arrivals in the same way, by pruning the first available slots of each car so as to fully charge it. Our problem can thus be fully described by the following integer programming (IP) formulation:
\begin{align}\label{IP}
\text{max }~ &\sum_{i \in \V} \sum_{j \in \J} \sum_{t \in \T} p_{j,t} x_{i,j,t}\\
\text{s.t. }~ &\sum_{i \in \V} x_{i,j,t} \leq 1,~~~~~~~~~~~~~~~~~~~~~~~ \forall j \in \J, t \in \T\label{eq:IP_c1}\\
 &\sum_{j \in \J} \sum_{t' \in [t, t+C_i] }  x_{i,j,t'} \leq 1,~~~~~~~~~~ \forall i \in \V, t \in \T_i\label{eq:IP_c2}\\
 &x_{i,j,t} \equiv 0,~~~~~~~~~~~~~~~~~~~ \forall i \in \V, j \in \J, t \notin \T_i\\
 &x_{i,j,t} \in \{0,1\},~~~~~~~~~~~~~~\forall i\in \V, j \in \J, t \in \T_i.
\end{align}
In the above IP, $x_{i,j,t}$ denotes the decision variable of discharging vehicle $i \in \V$ at station $j \in \J$ and at time $t \in \T$. Constraint \eqref{eq:IP_c1} ensures that every station can be used to discharge at most one vehicle at any time period, while constraint \eqref{eq:IP_c2} is the validity constraint for the recharging restriction, capturing that for every vehicle $i$, discharging cannot occur sooner than $C_i$ units of time following the previous discharge.

\subsection{Hardness result}
\begin{figure*}
\centering
\begin{minipage}{.8\textwidth}
  \centering
  \includegraphics[width=.9\linewidth]{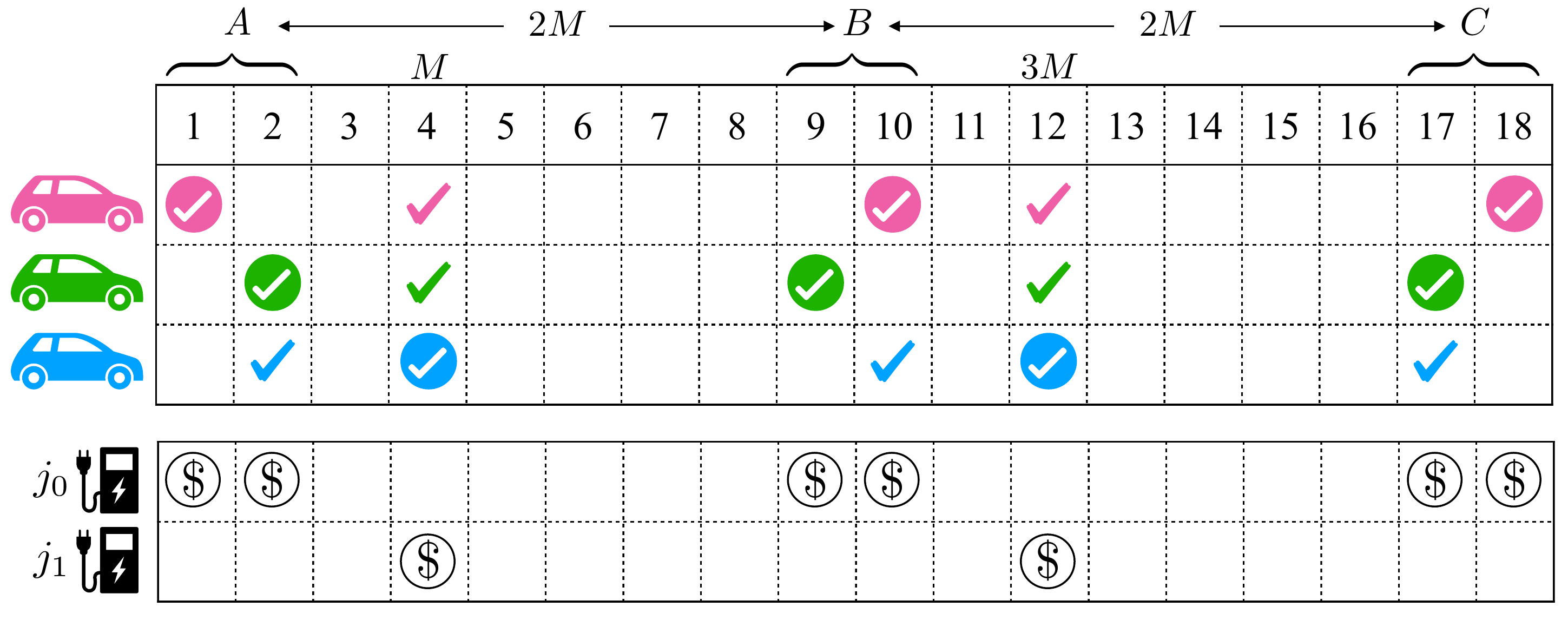}
\end{minipage}%
\begin{minipage}{.2\textwidth}
  \centering
  \includegraphics[width=.9\linewidth]{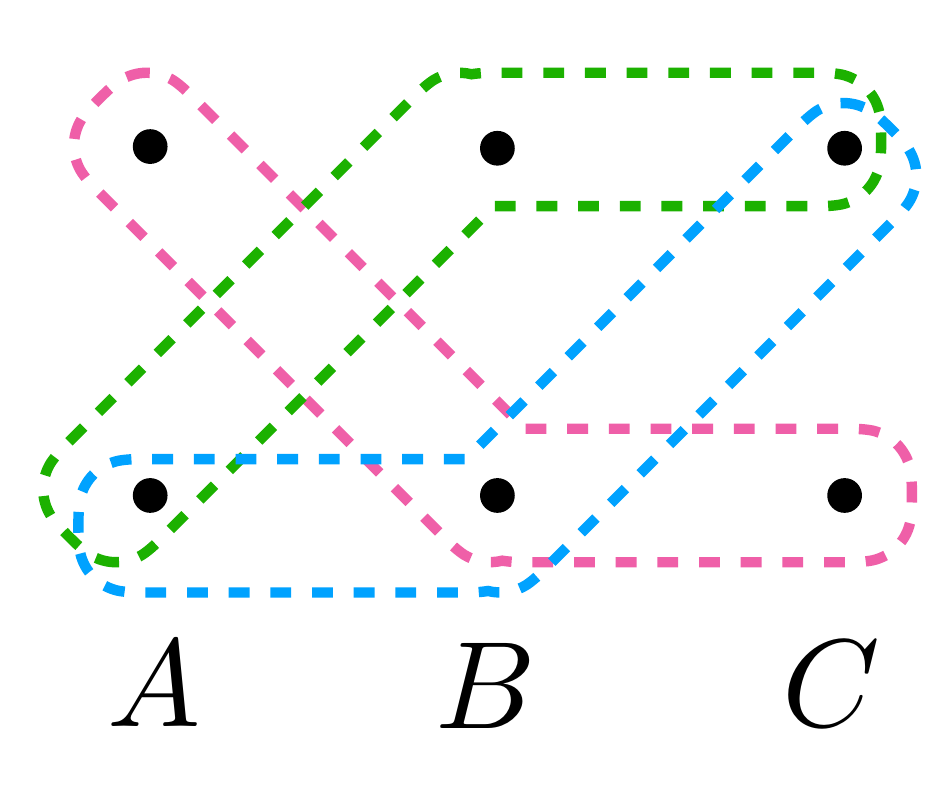}
\end{minipage}
\caption{Example with $k=2$, $M=4$, and $C=6$, created from the \textsc{3D-matching} instance on the right. Availabilities of each car (which encode the hyperedge of the same color) are checked in the upper table, and the circled ones represent the optimal solution which collects all the rewards in the lower table.}
\label{fig:reduction}
\end{figure*}
We now prove that for a non-constant number of vehicles and for arbitrary availability intervals, problem~\eqref{IP} is {\em strongly-NP hard}. In terms of algorithmic design, this statement implies that, unless $\textbf{P = NP}$, not only is there no polynomial-time algorithm for solving the problem optimally, but also there is no algorithm that can produce an assignment of reward greater than $1-\epsilon$ of the optimal, in time polynomial in the input size and in $1/{\epsilon}$.
\begin{theorem}
Problem~(\ref{IP}) is strongly NP-hard, even for the (easier) special case where all charging times are identical, all  rewards are 0-1 and there is only a single discharging station.
\end{theorem}
\vspace{-.3em}
\begin{proof}
In order to simplify the proof, we first consider the case of an arbitrary number of discharging stations and then we show how this can be extended to the single station case. In order to accomplish this, we use a reduction from the well-known strongly NP-hard \textsc{3D-matching} problem \cite{garey1979computers}. In the \textsc{3D-matching} problem, we are given as input three sets of nodes $A,B,C$ such that $\abs{A}= \abs{B} = \abs{C}$ and a set $E$ of hyperedges of the form $e_i = (\alpha_i, \beta_i, \gamma_i)$, where nodes $\alpha_i \in A, \beta_i \in B, \gamma_i \in C$. The problem is to decide whether there exists a matching $S \subseteq E$ such that every node of $A\cup B\cup C$ is covered by some edge and no two edges of $S$ share a node. 

Given an instance $\mathcal{I} = (A,B,C,E)$ of the \textsc{3D-matching} problem, we begin by considering a time horizon $\T$ such that $\T = [4M+k]$, where $k=\abs{A}= \abs{B} = \abs{C}$ and $M$ is a large number such that $M \geq 2 k$. For example, an instance of the \textsc{3D-matching} problem with $k=2$ is shown in Fig.~\ref{fig:reduction}~(right) with its corresponding time horizon on the left (assuming $M=4$). We create one station $j_0$ which has reward $p_{j_0, t} = 1$ for times $t \in [1,k] \cup [2M+1, 2M+k] \cup [4M+1, 4M+k]$, and $p_{j_0, t} = 0$ for the rest of the times. Moreover, we create $\abs{E}-k$ additional identical stations, each having reward $p_{j,t} = 1$ for $t \in \{M, 3M\}$, and $p_{j,t} = 0$, elsewhere. For the example of Fig.~\ref{fig:reduction}, since there are $\abs{E}=3$ hyperedges, we have only one additional station $j_1$.

Regarding the vehicles, we create a vehicle $i$ for each hyperedge $e_i = (\alpha_i, \beta_i, \gamma_i)$ and set $\T_i = \{\alpha_i, 2M+\beta_i, 4M+\gamma_i\} \cup \{M,3M\}$, where we assume that $\alpha_i, \beta_i, \gamma_i \in [k]$ are the indices of the nodes in an arbitrary (fixed) ordering. For example, in Fig.~\ref{fig:reduction} the blue hyperedge has $e=(\alpha=2,\beta=2,\gamma=1)$, assuming that the ordering in $A,B,C$ is from top to bottom. Therefore the corresponding EV has availability $\T=\{2,10,17\}\cup\{4,12\}$. Finally, for all vehicles, we set the same charging time $C = M+k$. 

By the above construction, given a polynomial-time exact algorithm for our problem, one could decide on the feasibility of \textsc{3D-matching}. More specifically, our algorithm is able to collect all the non-zero rewards if and only if the answer of the \textsc{3D-matching} instance is \textsc{yes}. Indeed, the vehicles that correspond to the $k$ hyperedges of the matching can collect every reward of the $j_0$ station, while the remaining $|E|-k$ vehicles can collect all the rewards in the $j \neq j_0$ stations. For the other direction, it can be shown that in the case where the answer of the $\textsc{3d-matching}$ problem is \textsc{no}, there is no feasible assignment of vehicles that can collect all the rewards.

Note that a similar construction can be made by using a single discharging station. Specifically, instead of adding $|E| - k$ stations with only two times of non-zero reward (recall $t \in \{M, 3M\}$), one could extend the time horizon of the single station $j_0$ and fit $|E| - k$ non-zero reward times between the times that correspond to the sets $A$ and $B$, and another group of  $|E| - k$ non-zero rewards between the times that correspond to $B$ and $C$. Then, for a proper choice of $M$, and thus a large enough charging time, the resulting five groups of times (i.e. $|A|, |E|-k, |B|, |E|-k, |C|$) can be made pairwise incompatible for every single vehicle (hyperedge).
\end{proof} 

\section{Special tractable cases}
\label{sec:special}
Given that the problem at hand is strongly NP-hard, we resort to approximation algorithms. Before that, to develop a better understanding and insight into the problem, we analyze several special cases that can be solved optimally in polynomial time.

\subsection{Zero charging time}
When we have access to super-fast chargers, which can charge the battery in a time much smaller than our time unit, we can assume that $C_i=0$. This implies that any vehicle can be discharged during any time of its availability, and we show that our problem can be reduced to the \textsc{maximum bipartite matching} problem \cite{Kleinberg:2005:AD:1051910}, which can be solved in polynomial time. To see this, consider a set of vertices $L$ such that $(i,t) \in L$ for any vehicle $i \in \V$, $t \in \T_i$, and a set of vertices $R$ such that $(j,t) \in R$ for all stations $j \in \J$ and times $t \in T$. We also define the set of edges $E = \{\{(i,t), (j,t')\}, \forall (i,t) \in L, \forall (j,t)\in R, t'=t\}$ and associate every edge that is adjacent to some node $(j,t) \in R$ with {\em weight} equal to the reward $p_{j,t}$. Given this construction, 
the optimal schedule corresponds to a maximum bipartite matching in the aforementioned graph, which can be solved optimally in polynomial time~\cite{Kleinberg:2005:AD:1051910}.

\subsection{Single vehicle}
The problem accepts a polynomial time LP-based algorithm for the case where there is a single vehicle. In this case, we can assume without loss of generality that there is only a single discharging station, otherwise we consider the station of highest reward at every time. We should note that, although the case of a single vehicle is contained in the case of constant number of vehicles (studied next), the following result provides useful intuition on the geometry and polyhedral aspects of the problem.
Let $x_t \in \{0,1\}$ be the decision variable denoting the discharging of the single car at time $t$. Consider the following LP relaxation of our problem:
\begin{align}\label{LP}
\text{max }~ &\sum_{t \in \T} p_{t} x_{t}\\
\text{s.t. }~
  &\sum_{t' \in [t, t+C_1]} x_{t'} \leq 1 ,~~~~~~~~~~~~~~~~\forall t \in \T_1 \label{eq:LP_c}\\
  &x_{t} \equiv 0,~~~~~~~~~~~~~~~~~~~~~~~~~~~~~\forall t \notin \T_1\\
  &x_{t} \geq 0,~~~~~~~~~~~~~~~~~~~~~~~~~~~~~\forall t \in \T,
\end{align}
We can show the following integrality result:
\begin{theorem}
\label{thm:integral}
The solution returned by the above LP relaxation~(\ref{LP}) is integral.
\end{theorem}
Since the LP relaxation provides an upper bound, the above integrality result implies that the solution returned by the LP is the optimal solution to the original problem.

\vspace{-0.5em}
\subsection{Constant number of vehicles}

The third class of instances that can be solved in polynomial time corresponds to the case where the number  of vehicles ($m$) is small enough and can be considered constant. For this case, we can solve the problem using dynamic programming. More specifically, we construct a matrix $OPT\in \mathbb{N}^{m+1}$ with entries $OPT(t,r_1, r_2, \dots, r_m)$, where every variable $r_i$ denotes the remaining time for a vehicle $i \in \V$ until it is able to be discharged again. Notice that the total size of this matrix is $\bigO(T C_{\max}^m)$, where $C_{\max} = \max_{i \in \V} C_i = \Theta(T)$. In general, for the computation of every element $OPT(t,r_1, r_2, \dots, r_m)$ of the matrix, we want to find which combination of available vehicles we can discharge such that the reward collected plus the future optimal reward given this discharging is maximized. Notice that trying all possible combinations of vehicles to discharge is polynomial, given the assumption of constant $m$. By the above analysis, it suffices to recursively compute all the entries of the $OPT(t,r_1, r_2, \dots, r_m)$ matrix, starting from the end of the time where the solution is trivial. After the computation of the whole matrix, the value $OPT(t=1, r_1 = 0, r_2 = 0, \dots, r_m = 0)$ would be the answer to our problem. Notice that we can easily extend the above algorithm to keep track of our best choices at each time and return, in addition to the optimal reward, the optimal schedule. In the following, we give an example of a recursive computation for the case of a single station:
\begin{multline*}
OPT(t, r_1, r_2, \dots, r_m) =\\
\max_{i\in \V\cup \{\emptyset\}}\Big\{OPT(t+1,C_i,r_{\overline{i}}-1)+p_{j,t}.\mathbf{1}{[t\in \T_i]}\Big\},
\end{multline*}
where $OPT(t+1,C_i,r_{\overline{i}}-1)$ is the short form for $OPT(t+1, r_1, r_2, \dots, r_m)$ where we replace $r_i$ with $C_i$ and all other indices (denoted by $\overline{i}$) are decremented by one (if they are positive).

\subsection{Homogeneous vehicles}

A similar dynamic programming approach works for the case where all vehicles have the same availability period and the charging time, say $C$, is constant and identical. In that case, we only care about the number of available vehicles and not their identity. We construct a matrix $OPT\in \mathbb{N}^{C+2}$, with elements $OPT(t, r_0, r_1, r_2, \dots, r_C)$, where $t \in \T$ is the time and $r_\ell$ is the number of vehicles that will be available in $\ell$ time periods. Clearly, $\ell$ is upper bounded by $C$, the charging time. Given this, the size of the matrix is $\bigO(T m^{C+1})$, and the elements can again be computed in a recursive manner as follows: 
\begin{multline*}
OPT(t, r_0, r_1, \dots, r_C) =\\
\max_{0\leq k\leq r_0} \Big\{OPT(t+1,r_0+r_1-k,r_2,\dots,r_C,k)\\
+\max_{\abs{S}=k} \sum_{j\in S} p_{j,t}.\mathbf{1}{[t\in \T_1]}\Big\} ,
\end{multline*}
where $k$ is the number of vehicles that we schedule at time $t$, upper bounded by the number of available cars $r_0$. For the optimal $k$, the set $S$ corresponds to the top $k$ stations with the highest rewards, which will be assigned to those $k$ vehicles. In the next time step $t+1$, these cars will appear as the last argument, since they now need $C$ time steps to get charged. In addition, all other arguments will be shifted one to the left, as the cars which needed $\ell$ time periods (for charging) at time $t$ will need another $\ell-1$ time steps at time $t+1$. 

\section{Approximation algorithms}
\label{sec:approx}
In this section, we present approximation algorithms for the general valet problem~\eqref{IP}.

\subsection{A $\frac{1}{3}$-approximation greedy algorithm}
We first present a greedy algorithm. Even though this algorithm has a weaker approximation guarantee than the algorithm in the following section, it has the advantage that it applies to a more general setting where the rewards can depend not only on the time and location, but also on the vehicle. That is, for any choice of vehicle $i \in \V$, station $j \in \J$ and time $t \in \T_i$ the collected reward can be vehicle specific, denoted by $p_{i,j,t}$. We define $\F = \{(i,j,t), i \in \V, j \in J, t \in \T_i\}$ to be the set of all feasible assignments, that is, all the possible triplets of vehicles to station/time pairs that can be realized individually, assuming an empty schedule. Consider the following algorithm: at each iteration, we greedily pick the highest available reward, assign it to an available car $i$ (or the car that gives that reward in the case of vehicle-dependent rewards), and remove all the triplets that are no longer feasible. This includes any other car that wants to discharge at the same location and time, as well as the triplets that wish to discharge the same car in the next $C_i$ time intervals. The formal algorithm is tabulated under Algorithm \ref{alg:greedy}. We can prove the following performance guarantee for the algorithm: 

\begin{algorithm}[t]
\caption{\textsc{greedy}}
\begin{algorithmic}[1]
\State Initialize: $\F = \{(i,j,t) | i \in \V, j \in \J, t \in \T_i\}$.
\While{$\F \neq \emptyset$}
\State Choose the feasible assignment of maximum reward:  \hspace*{1.4em}$(i,j,t) \leftarrow \text{argmax}\{p_{j,t} | (i,j,t) \in \F \}$.
\State Assign vehicle $i$ to the station/time $(j,t)$ and collect \hspace*{1.4em}the reward $p_{j,t}$.
\State Remove from the feasible set $\F$ all the assignments \hspace*{1.4em}that are no longer feasible. 
\EndWhile
\end{algorithmic}
\label{alg:greedy}
\end{algorithm}
\begin{theorem}
The \textsc{greedy} algorithm produces in polynomial time a schedule of total reward $P \geq \frac{1}{3} OPT$, where $OPT$ is the reward of an optimal schedule.
\end{theorem}
\begin{proof}
The polynomial running time of the algorithm can be trivially justified. The algorithm sorts $\bigO( m n T)$ initial assignments and then constantly updates the ordered list in linear time, by removing any infeasible assignments. The total number of iterations is at most $\bigO(\min{(n T, m T)})$. On the approximation ratio of the algorithm, it suffices to make the following crucial observation: Let $A$ and $O$ be the set of assignments in the \textsc{greedy} and the optimal solution, respectively. For any assignment $(i,j,t) \in A$, we distinguish between two cases: whether (a) it is in the optimal solution $(i,j,t)~\in~O$ or (b) it is not, i.e., $(i,j,t) \notin O$. Notice that in the second case, the assignment $(i,j,t)$ excludes at most three assignments that are made in the optimal solution: (b1) $\exists i' \neq i$ s.t. $(i',j,t)\in O$, that is, another vehicle is assigned to the same station/time pair in the optimal solution, (b2) there exist $t' \neq t$ and $j'$ such that $(i, j', t') \in O$ and $t' \in [t-C_i , t+C_i]$, i,e, the optimal solution contains an assignment of the same vehicle $i$ that is incompatible with $(i, j, t)$ given the charging time restriction. It is not hard to verify that for any vehicle $i$, there can be at most two assignments in any interval $[t-C_i , t+C_i]$. Therefore, given that \textsc{greedy} always chooses the assignment of highest reward, the reward of every $(i,j,t) \in A$, $p_{j,t}$, is greater than $\frac{1}{3}$ times the sum of rewards of the assignments of the optimal solution it excludes. Therefore, we have: 
\begin{align*}
P &= \sum_{(i,j,t) \in A} p_{j,t} \\
&= \sum_{(i,j,t) \in O \cap A} p_{j,t} +  \sum_{(i,j,t) \in A \backslash O} p_{j,t} \\ 
&\geq \sum_{(i,j,t) \in O \cap A} p_{j,t} + \frac{1}{3} \sum_{(i,j,t) \in O \backslash A} p_{j,t} \\
&\geq \frac{1}{3} \bigg[\sum_{(i,j,t) \in O \cap A} p_{j,t} + \sum_{(i,j,t) \in O \backslash A} p_{j,t} \bigg]\\
&= \frac{1}{3} \sum_{(i,j,t) \in O} p_{j,t} = \frac{1}{3}OPT.
\end{align*}
\vspace{-1.5em}
\end{proof}

\subsection{A $(1-\frac{1}{e})$-approximation randomized algorithm}
In this section, we present an LP-based randomized algorithm with a performance guarantee of $1 - \frac{1}{e}\approx 0.63$, which is better than the previous greedy algorithm. We consider the linear programming relaxation of IP~(\ref{IP}), which we create by allowing the decision variables to take any non-negative value, i.e. $x_{i,j,t} \geq 0$. This LP formulation gives an upper bound to the optimal solution of our problem. Since the LP solution in not necessarily integral (or even half-integral), we adapt a randomized rounding algorithm, proposed in \cite{bhatia2007algorithmic} for an interval scheduling problem with application to bandwidth trading. Although our problem is not a special case of the bandwidth trading problem \cite{bhatia2007algorithmic}, they are both special cases of a more general problem, which can be solved by the same randomized rounding algorithm. The formal description of the algorithm in the context of our problem is given under Algorithm~\ref{alg:RR}. 
\begin{algorithm*}[t]
\caption{\textsc{Randomized Rounding}}
\label{alg:RR}
\begin{algorithmic}[1]
\State Solve the LP relaxation and let $\{x_{i,j,t}\}$ be the (fractional) solution.
\For{every vehicle $i \in \V$}
\State Let $A_i = \{(j,t)| j \in \J, t \in \T_i, x_{i,j,t}>0\}$
\State Consider an empty 2-dimensional area: $[1,T+1] \times [0,1] \subseteq \mathbb{R}^2$ defined by the x-axis and the y-axis.
\For{$(j,t) \in A_i$ in non-decreasing order of $t \in \T$}
\State  Create a rectangular area of y-coordinates within $[0,x_{i,j,t}]$ and of x-coordinates within $[t, t+C_i+1)$, if the \hspace*{3em}corresponding area is empty. 
\State If this area is not empty, split the rectangle into smaller stripes of fixed width $[t, t+C_i+1)$, and fit these slices \hspace*{3em}anywhere in the y-axis.
(Notice that the rectangles can be fragmented only in the y-axis, but no fragmentation is \hspace*{3em}allowed in the x-axis that corresponds to the time.) 
\EndFor
\State Sample a horizontal line from $y=0$ to $y=1$ uniformly at random.
\State Assign vehicle $i \in \V$ to all the stations and times which are crossed by the line.
\EndFor
\State If more than one vehicle has been assigned in the same $(j,t), j \in \J, t \in \T$, arbitrarily keep one.
\end{algorithmic}
\end{algorithm*}

Since the LP solution can be fractional, a single car could be assigned fractionally to conflicting discharging stations. But constraint \eqref{eq:IP_c2} guarantees that at each time, the total assignment is at most one, allowing us to fit the assignments in a strip of height one as shown in  Fig.~\ref{fig:RR}. The rounding phase consists of sampling a random horizontal line in this figure, and picking the assignments that it intersects.
\begin{figure}[t]
\centering
\includegraphics[scale=0.7]{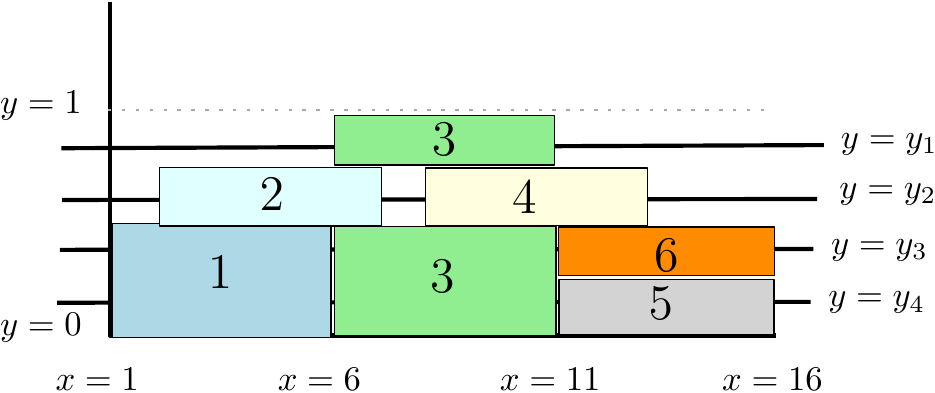}
\caption{Example of \textsc{Randomized Rounding} for a fixed vehicle $i$ of $C_i = 4$. We have the following non-zero assignments: $x_{i,1,1} = 0.50$, $x_{i,2,2} = 0.25$, $x_{i,3,6} = 0.75$, $x_{i,4,8} = 0.25$, $x_{i,5,11} = 0.25$ and $x_{i,6,11} = 0.25$. 
Notice that the rectangle of station 3 has been fragmented into two slices of height 0.25 and 0.50. 
For the random horizontal lines $y = y_1$ to $y = y_4$, the corresponding feasible assignments are $\{3\}$, $\{2,4\}$, $\{1,3,6\}$ and $\{1,3,5\}$, respectively.}
\label{fig:RR}
\end{figure}



The following theorem provides a performance guarantee for the randomized rounding algorithm. We refer the reader to \cite{bhatia2007algorithmic} for a complete proof.
\begin{theorem}[\cite{bhatia2007algorithmic}]
The \textsc{Randomized Rounding} algorithm creates a feasible schedule of expected reward $\mathbb{E}[P] \geq (1-\frac{1}{e}) OPT$, where $OPT$ is the reward of an optimal assignment.
\end{theorem}

\section{Simulation results}
\label{sec:numerical}
In this section, we evaluate the empirical efficiency of our proposed algorithms. We compare the collected reward with the optimal reward, for the instances where we can calculate the latter, and with an upper bound on the optimal reward, for larger instances where solving the integer program cannot be done within a reasonable amount of time. In the absence of actual data on our problem and in order to prove the robustness of our algorithms, we choose test cases of variable charging times, variable availability intervals and rewards that are generated by random initial starting values and then randomly change at each time step within a defined range.
\begin{figure*}
\centering
\includegraphics[width=1\linewidth]{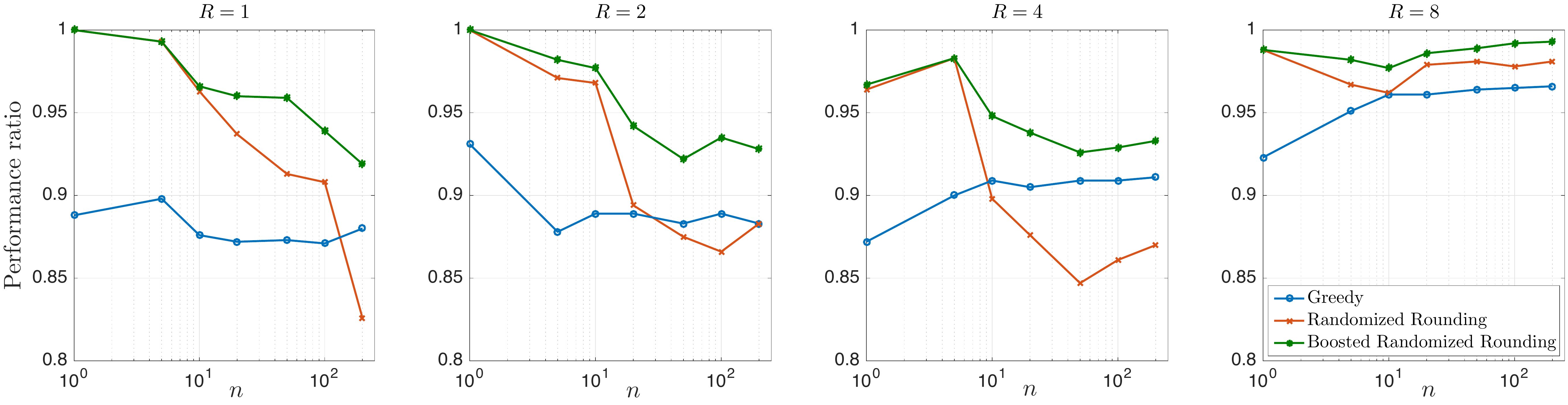}
\vspace{-1.5em}
\caption{Performance ratio of the proposed algorithms for different values of $R=m/n$.}
\label{fig:results}
\end{figure*}
\subsection{Simulation setting}
We compare the performance of our algorithms \textsc{greedy} and \textsc{randomized rounding} with the optimal reward \textsc{opt} (which we obtain by solving the IP) or with an upper bound on \textsc{opt} (which we obtain by solving the LP relaxation). We also consider the {\em boosted} version of the randomized algorithm, called \textsc{boosted randomized rounding}, in which we repeat the randomized algorithm 10 times and return, among the solutions we found, the one with the highest reward, amplifying in this way the probability of choosing a profitable schedule.

In terms of test cases, we consider a time horizon of $T = 24$ hours and various combinations of the number of vehicles and stations. The charging time of each vehicle is a number between 1 and 6 hours (to capture different battery sizes, charging rates, etc.) chosen uniformly at random (u.a.r.). On the availability of each vehicle, we distinguish between two types of customers of equal probability: those with a single continuous interval of length 1 to 24 chosen u.a.r., and those with three (smaller) intervals of length 1 to 8 chosen u.a.r..  We also define the starting time of each interval to be a number between 1 and 24 chosen u.a.r.. On the choice of rewards, we initially choose for each station $j \in \J$ a value $p_{j,1}$ between 0 and 100 u.a.r. to denote the average reward, differentiating in this way the stations in terms of their location. Then, we model a smooth change in the reward on a specific station as follows: for $t>1$ we draw $p_{j,t}$ from $[\ell_{j,t},u_{j,t}]$ u.a.r. with $\ell_{j,t} = \max\{0.7 \cdot p_{j,t-1},p_{j,1}-25,0\}$ and $u_{j,t} = \min\{1.3 \cdot p_{j,t-1},p_{j,1}+25,100\}$. That is, the reward of any (hour, station) pair is always within $70\%$ to $130\%$ of the reward of the previous hour, and always within a band of $\pm 25$ from the initial reward.

Given our observation that the number of vehicles is without loss of generality greater that the number of stations, we simulate for different values of the following two parameters: (a) The number $n$ of discharging stations and (b) the {\em ratio} of the number of vehicles $m$, over the number of stations $n$, denoted by $R = \frac{m}{n}$. The values we consider are $n \in \{1, 5, 10, 20, 50, 100, 200\}$ and $R \in \{1, 2, 4, 8\}$. Finally, for every choice of $n$ and $R$, we perform $10$ independent experiments and return, for each algorithm, the average of the empirical ratios found. 

We solve the integer and linear programs using Gurobi Optimizer 8.0, while we use Python 2.7 for scripting.

\subsection{Presentation and analysis of the results}
The empirical approximation ratios resulting from our simulations are presented in Table~\ref{table:results} and Fig.~\ref{fig:results}. We denote by star (*) the cases where the approximation ratio is computed with respect to the optimal reward.
\begin{table}[h]
\centering
\caption{Empirical approximation ratios.}
\vspace{-0.5em}
\label{table:results}
\scalebox{0.8}{
\begin{tabular}{| c | c | c | c | c | c | c | c |}
\hline
\textbf{R=1} & 1* & 5* & 10* & 20* & 50* & 100* & 200* \\ \hline
\textsc{G}& 0.888& 0.898& 0.876& 0.872& 0.873& 0.871& 0.880\\ \hline
\textsc{RR}& 1.000& 0.993& 0.963& 0.937& 0.913& 0.908& 0.826\\ \hline
\textsc{BRR}& 1.000& 0.993& 0.966& 0.960& 0.959& 0.939& 0.919\\ \hline
\end{tabular}}
\scalebox{0.8}{
\begin{tabular}{| c | c | c | c | c | c | c | c |}
\hline
\textbf{R=2} & 1* & 5* & 10* & 20* & 50* & 100* & 200* \\ \hline
\textsc{G}& 0.931& 0.878& 0.889& 0.889& 0.883& 0.889& 0.883\\ \hline
\textsc{RR}& 1.000& 0.971& 0.968& 0.894& 0.875& 0.866& 0.883\\ \hline
\textsc{BRR}& 1.000& 0.982& 0.977& 0.942& 0.922& 0.935& 0.928\\ \hline
\end{tabular}}
\scalebox{0.8}{
\begin{tabular}{| c | c | c | c | c | c | c | c |}
\hline
\textbf{R=4} & 1* & 5* & 10* & 20* & 50* & 100* & 200* \\ \hline
\textsc{G}& 0.872& 0.900& 0.909& 0.905& 0.909& 0.909& 0.911\\ \hline
\textsc{RR}& 0.964& 0.983& 0.898& 0.876& 0.847& 0.861& 0.870\\ \hline
\textsc{BRR}& 0.967& 0.983& 0.948& 0.938& 0.926& 0.929& 0.933\\ \hline
\end{tabular}}
\scalebox{0.8}{
\begin{tabular}{| c | c | c | c | c | c | c | c |}
\hline
\textbf{R=8} & 1* & 5* & 10* & 20* & 50* & 100 & 200 \\ \hline
\textsc{G}& 0.923& 0.951& 0.961& 0.961& 0.964 & 0.965 & 0.966\\ \hline
\textsc{RR}& 0.988& 0.967& 0.962& 0.979& 0.981 & 0.978 & 0.981\\ \hline
\textsc{BRR}& 0.988& 0.982& 0.977& 0.986& 0.989 & 0.992 & 0.993\\
\hline
\end{tabular}}
\end{table}

In Fig.~\ref{fig:results} we have plotted the performance ratio of the proposed algorithms separately for each value of $R$, the ratio between the number of cars and stations, as a function of $n$. 
Nevertheless, there is no clear relation between the number of stations ($n$) on the $x$-axis and the performance ratio on the $y$-axis (remember that our worst-case performance ratios are constant). Hence, these plots should be merely seen as a comparison between the proposed algorithms in different scenarios.
Based on our simulation results, we can make the following observations: 
\begin{itemize}
\item The actual performance of our algorithms in every case is significantly better than the theoretically proven worst-case guarantees, i.e. $82\% - 100\%$ of the optimal. Moreover, the overall performance of all algorithms increases with the ratio $R$. 
\item The \textsc{greedy} algorithm  (which has the worst approximation ratio of $\frac{1}{3}$), has the worst empirical performance among our algorithms, specifically $87\%-96\%$ of the optimal. 

\item \textsc{boosted randomized rounding} achieves the best performance, producing empirical ratios of at least $92\%$ of the optimal. Moreover, it produces, as we expected, significantly better results than the simple \textsc{randomized rounding}.

\item For the case of $R=1$ and $n=1$, we observe that the \textsc{randomized rounding} always produces the optimal result, a fact that is justified by the integrality of the LP relaxation in this case (see Theorem~\ref{thm:integral}). 
\end{itemize}


\section{Conclusion}
\label{sec:conclusion}
We have introduced the EV Valet concept, a new method for utilizing EVs as dynamic mobile energy storage to provide grid load balancing. 
The method identifies and builds on a symbiosis between the goals of car owners to get free parking in high-cost inner city areas, and the desire of utility companies to alleviate congestion overloads and to delay otherwise required expensive upgrades.  This symbiosis allows for a third party valet company to act as an intermediary between the two enabling this mutually beneficial relationship.

In addition to introducing this novel idea, we develop the necessary theory for the effective operation of this proposed valet system.  Specifically, we show that the valet problem is strongly NP-hard, yet we are able to identify cases that can be solved optimally in polynomial time.  Moreover, we solve the general problem using approximation algorithms, which empirically perform significantly better than their worst-case theoretical guarantees that we provide. 

There are numerous open computational directions and model extensions, such as: (i) whether better approximation algorithms can be developed; (ii) how the optimal scheduling can be done if the EV arrivals and departures are unknown in advance; (iii) how to hire the optimal number of valet employees for an extended time horizon such as a season.  

Lastly, an important complementary extension of our work is to design an optimal mechanism for the rewards the utility company should offer so as to align the incentives of the valet company with its own goals.


%





\ifCLASSOPTIONcaptionsoff
  \newpage
\fi



%
\bibliographystyle{IEEEtran}
\bibliography{ref.bib}







\end{document}